\documentclass[reqno,oneside,english]{amsart}
\usepackage{mathptmx}

\usepackage[T1]{fontenc}
\usepackage[latin9]{inputenc}
\usepackage{verbatim}
\usepackage{mathrsfs}
\usepackage{amsthm}
\usepackage{amssymb}
\usepackage{esint}
\usepackage{color}
\usepackage{enumerate}

\makeatletter
\numberwithin{equation}{section}
\numberwithin{figure}{section}
\theoremstyle{plain}
\newtheorem{thm}{\protect\theoremname}[section]
  \theoremstyle{plain}
  \newtheorem{lem}{Lemma}[section]
  \theoremstyle{definition}
  \newtheorem{defn}[thm]{\protect\definitionname}
  \theoremstyle{plain}
  \newtheorem{rem}[thm]{\protect\remarkname}
  \theoremstyle{remark}
  
   \theoremstyle{plain}
  \newtheorem{hyp}{Hypotheses}[section]
  \theoremstyle{plain}
  \newtheorem{prop}[thm]{\protect\propositionname}
  \theoremstyle{plain}

\usepackage{bbm}

\makeatletter

\def\renewtheorem#1{%
  \expandafter\let\csname#1\endcsname\relax
  \expandafter\let\csname c@#1\endcsname\relax
  \gdef\renewtheorem@envname{#1}
  \renewtheorem@secpar
}

\def\renewtheorem@secpar{\@ifnextchar[{\renewtheorem@numberedlike}{\renewtheorem@nonumberedlike}}
\def\renewtheorem@numberedlike[#1]#2{\newtheorem{\renewtheorem@envname}[#1]{#2}}
\def\renewtheorem@nonumberedlike#1{  
\def\renewtheorem@caption{#1}
\edef\renewtheorem@nowithin{\noexpand\newtheorem{\renewtheorem@envname}{\renewtheorem@caption}}
\renewtheorem@thirdpar
}

\def\renewtheorem@thirdpar{\@ifnextchar[{\renewtheorem@within}{\renewtheorem@nowithin}}

\def\renewtheorem@within[#1]{\renewtheorem@nowithin[#1]}
\makeatother

\renewtheorem{claim}{Hypothesis}[section]

\DeclareMathOperator{\Tr}{Tr}

\makeatother

\usepackage{babel}
  \providecommand{\claimname}{Claim}
  \providecommand{\corollaryname}{Corollary}
  \providecommand{\definitionname}{Definition}
  
  \providecommand{\propositionname}{Proposition}
  \providecommand{\remarkname}{Remark}
  \providecommand{\theoremname}{Theorem}

\allowdisplaybreaks

\begin{document}
\pagenumbering{arabic}

\global\long\def\alg{\mathscr{F}}
\global\long\def\R{\mathbb{R}}
\global\long\def\N{\mathbb{N}}
\global\long\def\Z{\mathbb{Z}}
\global\long\def\C{\mathbb{C}}
\global\long\def\Q{\mathbb{Q}}
\global\long\def\F{\mathcal{F}}
\global\long\def\L{\mathscr{L}}
\global\long\def\E{\mathbb{E}}
\global\long\def\R{\mathbb{R}}
\global\long\def\H{\mathcal{H}}
\global\long\def\Pr{\mathbb{P}}
\global\long\def\borel{\mathfrak{B}}
\global\long\def\Norm{\mathscr{N}}
\global\long\def\Rd{\mathbb{R}^{d}}
\global\long\def\Hsd{\mathscr{H}}
\global\long\def\I{I_{\delta}}
\global\long\def\u{\mathbbm{1}}
\global\long\def\norm#1{\left\Vert #1\right\Vert }
\global\long\def\scal#1#2{\left\langle #1,#2\right\rangle _{H}}
\global\long\def\bh{\bar{h}}

\title[characterizations of Sobolev spaces]{Characterizations of Sobolev spaces on sublevel sets in abstract Wiener spaces}

\author{Davide Addona, Giorgio Menegatti, Michele Miranda Jr}

\maketitle
\begin{abstract} 
In this paper we consider an abstract Wiener space $(X,\gamma,H)$ and an open subset $O\subseteq X$ which satisfies suitable assumptions. For every $p\in(1,+\infty)$ we define the Sobolev space $W_{0}^{1,p}(O,\gamma)$  as the closure of Lipschitz continuous functions which support with positive distance from $\partial O$ with respect to the natural Sobolev norm, and we show that under the assumptions on $O$ the space $W_{0}^{1,p}(O,\gamma)$ can be characterized as the space of functions in $W^{1,p}(O,\gamma)$ which have null trace at the boundary $\partial O$, or, equivalently, as the space of functions defined on $O$ whose trivial extension belongs to $W^{1,p}(X,\gamma)$.
\end{abstract}

\section{Introduction}

In this paper we consider an abstract Wiener space $(X,\gamma,H)$, i.e., $X$ is a real separable Banach space endowed with a centered non-degenerate Gaussian measure $\gamma$ and $H$ is the associated Cameron-Martin space, and a subset $O\subseteq X$ with $O=G^{-1}((-\infty,0))$, where $G$ is a function which satisfies suitable assumptions (see Hypotheses \ref{claim:regularity}).


The topic of Sobolev spaces $W^{k,p}(X,\gamma)$ in a Wiener space $X$ is well established (see e.g. \cite{Bog}), while Sobolev spaces in subsets of a Wiener space admits different definitions, and they have been treated for example in \cite {Cel}, \cite{Hin} and \cite{Hin4}.  Following \cite{Cel} (see also \cite{Add}), we consider $W^{1,p}(O,\gamma)$ as the domain of closure $\nabla_{H}$ of the $H$-gradient operator on Lipschitz continuous functions.

In \cite{Cel}, the set $O$ is the sublevel $G^{-1}((-\infty,0))$
of a function $G$. Under some regularity assumptions on $G$ it is possible to define a surface measure $\rho$ (Hausdorff-Gauss infinite dimensional measure or Feyel-de La Pradelle measure, see e.g. \cite{Fey}). Moreover, in \cite{Cel} the authors show the existence of a bounded operator (trace operator) $\Tr$ from $W^{1,p}(O,\gamma)$ to $L^{q}(\partial O,\rho)$ with $p>1$ and $q\in[1,p)$. Thanks to this operator, it is possible to introduce an integration-by-parts formula on $O$ which generalize that on the whole space. Namely, for every $\varphi\in W^{1,p}(O,\gamma)$ and every $h\in H$, we have 
\begin{align}
\label{int_by_parts_formula_intro}
\int_{O}(\partial_{h}\varphi-\hat{h}\varphi)\ d\gamma=\int_{\partial O}\Tr\varphi\frac{\scal{\nabla_{H}G}h}{\|\nabla_{H}G\|_{H}}\ d\rho,
\end{align}
with $\partial_{h}\varphi(x)=\left\langle \nabla_{H}\varphi(x),h\right\rangle _{H}$
and $\hat{h}=R_{\gamma}^{-1}(h)$, where $R_{\gamma}$ is the covariance
operator of $\gamma$.  Integration-by-parts formula as \eqref{int_by_parts_formula_intro} on domains have been also proved in \cite{Add,Bon} with different techniques. Strengthening the assumptions on $G$, the trace operator can be extended as an operator from $W^{1,p}(O,\gamma)$ onto $L^p(\partial O,\rho)$ for every $p>1$. Finally, in \cite{Cel} the authors prove that the subspace of $f\in W^{1,p}(O,\gamma)$, consisting of functions with null trace on $\partial O$, coincides with the subspace of $f\in W^{1,p}(O,\gamma)$, whose elements are those functions whose trivial extension to the whole $X$ belongs to $W^{1,p}(X,\gamma)$.

In this paper we consider $O=G^{-1}((-\infty,0))$ and we define the space $W^{1,p}_0(O,\gamma)$ as the closure, with respect to the $W^{1,p}(O,\gamma)$-norm, of Lipschitz continuous functions whose support has positive distance from $O^c$. Eventually we prove that, under suitable conditions on $G$, for every $p>1$, $W_0^{1,p}(O,\gamma)$ is the space of functions in $W^{1,p}(O,\gamma)$ with null trace on $\partial O$ (Theorem \ref{thm:approximation:W_0}).

Examples of spaces $O$ to which our results apply can be found in the Section
\ref{sec:Examples}: they include subgraphs of functions with some regularity,
and subsets of $X$ in the particular case in which $X$ is the
Wiener space which models the Brownian motion or in the case in which $X$ is the Wiener space which models the pinned Brownian motion. 

We stress that these examples may not comprehend many regular sets like balls, neither if $X$ is a Hilbert space. This limitation is strictly related to our approach, and also appears in \cite{Cel}, in the case when the operator ${\rm Tr}$ maps $W^{1,p}(O,\gamma)$ onto $L^p(\partial O,\rho)$ when $p>1$. To the best of our knowledge, nowadays there is no other result about the definition of a trace operator from $W^{1,p}(O,\gamma)$ onto $L^p(\partial O,\rho)$ for more general subsets $O$, and also the case $p=1$ is not reached. This is one of the main gap with respect to the finite dimension, where the theory of traces for Sobolev functions is well-understood and complete. For the case $p=1$, a possible alternative approach is to consider BV functions in open domains in Wiener spaces, which are studied and characterized in \cite{Add2}. However, it is still not clear how to extend the theory of traces for BV functions in finite dimension to this setting. 

Beside traces, another open question in infinite dimension is what domains $O$ allow the construction of extension operators for Sobolev functions. Again, the case when $O$ is an open ball is still an open problem, even when $X$ is a Hilbert space. A negative answer is given in \cite{Bog1}, where the authors provide an example of open convex subset $O$ of a Hilbert space $X$ such that, for every $p>1$, there exists a function $f\in W^{1,p}(O,\gamma)$ which does not admit a Sobolev extension to the whole $X$. On the contrary, an example of extension operator can be found in \cite{Add0}, where the authors show that if $O$ is an half-plane that it admits an extension operator, and explicitly provide such an extension.



\subsection*{Acknowledgements}
G. M.  wants to thank Michael R\"ockner for posing the problem which originated this work and for several important suggestions, and moreover for hosting him to Bielefeld University for a research period. D. A. claims that this research has financially been supported by the Programme ``FIL-Quota Incentivante'' of University of Parma and co-sponsored by Fondazione Cariparma.


\section{Notations and preliminary results}
\label{sec:Setting}


In the following, for any $k,d\in\N$ we denote by $C_{b}^{k}(\R^{d})$
the set of $k$-times differentiable functions on $\Rd$ with all
derivatives uniformly bounded. $C_{b}^{\infty}(\R^{d})$ is the set
of bounded smooth functions on $\Rd$ which belongs to $C_b^k(\Rd)$ for every $k\in\N$. $C_{c}^{\infty}(\R^{d})$ is the
set of functions in $C_{b}^{\infty}(\R^{d})$ with compact support.

For every real-valued function $f$ defined in a subset $A\subseteq X$, we denote by $\overline f$ its trivial extension on $X$, i.e., $\overline f=f$ on $A$ and $\overline f=0$ on $A^c$.

For every $A\subseteq X$, we denote by $\u_A$ the characteristic function of $A$, i.e., $\u_A(x)=1$ if $x\in A$ and $\u_A(x)=0$ if $x\notin A$.

Let $K$ be a real separable Hilbert space. 
and let $\mathcal L(K)$ be the space of linear bounded operators on $K$. 
We denote by $\mathcal L_2(K)$ the subspace of $\mathcal L(K)$ whose elements $L$ satisfy
\begin{align*}
\|L\|_{\mathcal L_2(K)}^2:=\sum_{n=1}^\infty\|Le_n\|_K^2<+\infty,    
\end{align*}
where $\{e_n:n\in\N\}$ is any orthonormal basis of $K$. The elements of $\mathcal L_2(K)$ are called Hilbert-Schmidt operators, and the norm $\|\cdot\|_{\mathcal L_2(K)}$ is the Hilbert-Schmidt norm. The space $(\mathcal L_2(K),\|\cdot\|_{\mathcal L_2(K)})$ is a separable Hilbert space if endowed with the inner product
\begin{align*}
[L,M]_{\mathcal L_2(K)}
= \sum_{n=1}^\infty\langle Le_n,Me_n\rangle_K, \qquad L,M\in \mathcal L_2(K), 
\end{align*}
where $\{e_n:n\in\N\}$ is any orthonormal basis of $K$. 
Given a real separable Banach space $X$, we denote by $\mathcal B(X)$ the Borel subsets of $X$.

We denote by ${\rm Lip}(X)$ the set of Lipschitz continuous functions from $X$ onto $\R$. For every open set $O\subseteq X$ we denote by ${\rm Lip}(O)$ the set of Lipschitz continuous functions on $O$, by ${\rm Lip}_b(O)$ the set of bounded Lipschitz continuous functions on $O$, and by ${\rm Lip}_c(O)$ the set of Lipschitz continuous functions on $O$ whose support has positive distance from $O^c$.
\medskip{}

We recall some definitions and properties of abstract Wiener spaces (see e.g. \cite{Bog}). Let $X$ be a separable Banach space, let $X^{*}$ be its dual and let $X^{**}$ be the dual of $X^{*}$. We will suppose that $\gamma$ is a centered non-degenerate Gaussian measure on $X$. 

We consider the embedding $j:X^{*}\hookrightarrow L^{2}(X,\gamma)$, and we define the reproducing kernel $X_{\gamma}^{*}$ as the closure in $L^{2}(X,\gamma)$ of $j(X^{*})$. It is a separable Hilbert space endowed with the $L^2$-norm, and we introduce the covariance operator $R_{\gamma}:X_{\gamma}^{*}\rightarrow X^{*}{}^{*}$ defined as 
\[
R_{\gamma}f(g)=\int_{X}fj(g) d\gamma, \qquad f\in X_{\gamma}^{*}, \ g\in X^{*}.
\]
$R_{\gamma}$  is injective, and its range is contained in $X$, by identifying $X$ with its natural embedding in $X^{**}$. We define the Cameron-Martin space
$H$ as $R_{\gamma}(X_{\gamma}^{*})\subseteq X$; $H$ inherits a structure of separable Hilbert
space through $R_{\gamma}$: we define $\left\langle \cdot,\cdot\right\rangle {}_{H}$
as the inner product in $H$ and $\left\Vert \cdot\right\Vert _{H}$ as the associated norm. As a subspace of $X$, $H$ is dense.
If $h\in H$, we define $\hat{h}=R_{\gamma}^{-1}(h)$, so that $\hat{h}\in X_{\gamma}^{*}\subseteq L^{2}(X,\gamma)$. The triple $(X,H,\gamma)$ is called abstract Wiener space.

We fix an orthonormal basis $\{h_{i}:i\in\N\}$ of $H$ such that $h_{i}\in R_{\gamma}(X^{*})$ for every $i\in\N$.
We have that $\{\hat{h}_{i}:i\in\N\}$ is an orthonormal basis of $X_{\gamma}^{*}\subseteq L^{2}(X,\gamma)$, and for every $f\in L^{2}(X,\gamma)$ we get
\begin{equation}
    \sum_{i=1}^{+\infty}|\langle f,\hat{h}_{i}\rangle_{L^2(X,\gamma)}|^2\leq\|f\|_{L^2(X,\gamma)}^2. \label{sum}
\end{equation}
For every $n\in\N$ we define the projection $\pi_n:X\rightarrow {\rm span}\{h_1,\ldots,h_n\}\subseteq H$ as
\[
\pi_{n}(x)=\sum_{i=1}^{n}\hat{h}_{i}(x)h_{i}, \qquad x\in X.
\]

We denote by $L^{p}(X,\gamma;H)$ the space of (equivalence classes
of) Bochner integrable functions $f:X\to H$ such that 
\[
\|f\|_{L^{p}(X,\gamma;H)}:=\bigg(\int_{X}\|f\|_{H}^{p}\ d\gamma\bigg)^{1/p}<\infty.
\]
$L^{p}(X,\gamma;H)$ is a Banach space endowed with the norm $\|\cdot\|_{L^{p}(X,\gamma,H)}$
(see e.g. \cite{Die}).

Let $n\in\N$ and let $F$ be a $n$-dimensional subspace of $R_{\gamma}(X^{*})\subseteq H$%
. If $\{e_{i},\ldots,e_{n}\}$ is an orthonormal basis of $F$, we define the projection $\pi_{F}$ of $X$ on $F$ as the bounded linear function
\[
\pi_{F}(x)=\sum_{i=1}^{n}\hat{e}_{i}(x)e_{i}
\]
for every $x\in X$. $\pi_{F}$ is uniquely defined, independently from the choice of the basis.

We denote by $\gamma_{F}$ the image measure $\gamma\circ\pi_{F}^{-1}$
on $F$, i.e.,
\[
\gamma_{F}(A)=\gamma(\pi_{F}^{-1}(A))
\]
for every $A$ Borel set in $F$. It follows that $\gamma_{F}$ is a non-degenerate centered Gaussian measure, and there exist a Banach space $X_{F^{\perp}}$ and a non-degenerate centered Gaussian measure $\gamma_{F^{\perp}}$ such that we have an isometry between $F\times X_{F^{\perp}}$ and $X$, and $\gamma=\gamma_{F}\otimes\gamma_{F}^{\perp}$. This is said factorization of $\gamma$ with respect to $F$.

We define the space of bounded infinitely many times differentiable
cylindrical functions $\mathcal{F}C_{b}^{\infty}(X)$ as the set of
functions $f:X\rightarrow\R$ such that 
\[
f(x)=g(l_{1}(x),\ldots,l_{n}(x)),\qquad x\in X,
\]
where $l_{1},\ldots,l_{n}\in X^*$ are bounded linear functions on $X$ and $g\in C_{b}^{\infty}(\R^{n})$ for some $n\in\N$. We recall that
$\mathcal{F}C_{b}^{\infty}(X)$ is dense in $L^{p}(X,\gamma)$ for every $p\in[1,+\infty)$. $\mathcal{F}C_{b}^{\infty}(X;H)$
denotes the set of functions $f:X\rightarrow H$ with finite dimensional range such that, for every $l\in H$, we have $x\mapsto \langle l, f(x)\rangle_H\in\mathcal{F}C_{b}^{\infty}(X)$. In particular, $\mathcal FC_b^\infty(X;H)$ is spanned by functions $\phi h$ with $\phi\in \mathcal FC_b^\infty(X)$ and $h\in H$. It is easy to prove that $\mathcal{F}C_{b}^{\infty}(X;H)$
is dense in $L^{p}(X,\gamma;H)$.

For every smooth function $f:X\rightarrow \R$, every $h\in H$ and every $x\in X$, we define the partial derivative $\partial_{h}f(x)$ of $f$ at $x$ along $h$
as 
\begin{equation}
\partial_{h}f(x):=\lim_{\varepsilon\rightarrow0}\frac{f(x+\varepsilon h)-f(x)}{\varepsilon}\label{eq:derivative}
\end{equation}
and the partial logarithmic derivative $\partial_{h}^{*}f(x)$ of $f$ at $x$ along $h$ as 
\[
\partial_{h}^{*}f(x):=\partial_{h}f(x)-f(x)\hat{h}(x).
\]
We say that $f$ is $H$-differentiable in $x\in X$ if there exists $\nabla_{H}f(x)\in H$ such that 
\[
\partial_{h}f=\left\langle \nabla_{H}f,h\right\rangle {}_{H}, \qquad h\in H.
\]

If $f\in\mathcal{F}C_{b}^{\infty}(X)$ then it is everywhere $H$-differentiable, $\nabla_Hf=R_\gamma Df$, where $Df$ is the Fr\'echet derivative of $f$, and $\nabla_{H}f\in L^{\infty}(X,\gamma;H)$.
Further, the operator $\nabla_{H}$ is well defined for any Lipschitz continuous function $f$ and $\nabla_{H}f\in L^{\infty}(X,\gamma;H)$ (see e.g.
\cite[Theo. 5.11.2]{Bog}). 

For every $p\in[1,+\infty)$, $\nabla_{H}:\mathcal {F}C_b^\infty(X)\rightarrow L^{p}(X,\gamma;H)$ is a closable operator in $L^{p}(X,\gamma)$. We still denote its closure as $\nabla_H$ and we define the Sobolev space $W^{1,p}(X,\gamma)$ as the domain of this closure (see \cite[Sec. 5.2]{Bog}).
Moreover, if $f\in W^{1,p}(X,\gamma)$, then $\nabla_{H}f\in L^{p}(X,\gamma;H)$ and
for every $h\in H\backslash\{0\}$ we set $\partial_{h}f=\left\langle \nabla_{H}f,h\right\rangle _{H}$. 

Let $f:X\rightarrow H$. We say that $f$ is $H$-differentiable at $x\in X$ if there exists a Hilbert-Schmidt operator $D_{H}f(x)$ on $H$ such that
\[
D_{H}f(x)h=\lim_{\varepsilon\rightarrow0}\frac{f(x+\varepsilon h)-f(x)}{\varepsilon}, \qquad x\in X, \ h\in H.
\]

For every $p\in[1,+\infty)$, $D_{H}:\mathcal {F}C_b^\infty(X;H)\rightarrow L^{p}(X,\gamma;\mathcal L_2(H))$ is a closable operator on $L^{p}(X,\gamma,H)$. We still denote by $D_H$ its closure and we define $W^{1,p}(X,\gamma;H)$ as the domain of
this closure (see \cite[Sec. 5.2]{Bog}).

Let $f:X\rightarrow\R$ be such that $\nabla_{H}$ is defined at each
point $x\in X$. We say that $f$ is twice $H$-differentiable at $x\in X$
if $\nabla_{H}f$ is $H$-differentiable at $x$. We set $D_{H}^{2}f(x):=D_H(\nabla_Hf)(x)$, and we recall that the operator $D^2_Hf(x):H\times H\rightarrow \R$ is a Hilbert-Schmidt operator on $H$. $D_{H}^{2}f(x)$ is said $H$-second
derivative of $f$ at $x$.

The operator $(\nabla_H,D_{H}^{2}):\mathcal{F}C_{b}^{\infty}(X)\rightarrow L^p(X,\gamma;H)\times L^p(X,\gamma;\mathcal L_2(H))$ is a closable operator on $L^p(X,\gamma)$ for every $p\in[1,+\infty)$. We denote by $W^{2,p}(X,\gamma$)
the domain of the closure of the operator $(\nabla_H,D^2_H)$ (see \cite[Sec. 5.2]{Bog}).

We recall the concept of $H$-divergence (see \cite[Sec. 5.8]{Bog}).
For every $f\in\mathcal{F}C_{b}^{\infty}(X;H)$ we define the $H$-divergence
$\mbox{div}{}_{\gamma}$ with respect to $\gamma$ as
\begin{equation}
\mbox{div}{}_{\gamma}f=\sum_{i=1}^{\infty}(\partial_{h_{i}}f_{i}-\hat{h_{i}}f_{i})=\sum_{i=1}^\infty \partial ^*_{h_i}f_i,
\label{eq:divergence}
\end{equation}
where $\{h_{i}:i\in\N\}$ is an orthonormal basis of $H$ and $f_{i}=\left\langle f,h_{i}\right\rangle {}_{H}$ for every $i\in \N$. The definition of $\mbox{div}{}_{\gamma}$ does not depend on the choice of the basis of $H$. Further, if $f:X\rightarrow H$ is everywhere $H$-differentiable
with $D_{H}f$ uniformly bounded, then $\mbox{div}{}_{\gamma}f$ is
defined everywhere (through formula (\ref{eq:divergence})). 

Let $f\in W^{1,2}(X;H)$. For every $n\in\N$ we define $f_{n}(x)=\pi_{n}\circ f(x)$ for every $x\in X$. It follows that the divergence $\mbox{div}{}_{\gamma}f_{n}$ is defined
$\gamma$-a.e. in $X$, it belongs to $L^{2}(X,\gamma)$ and it converges in $L^{2}(X,\gamma)$ to a function $g\in L^2(X,\gamma)$ which we denote by $\mbox{div}{}_{\gamma}f$ (see \cite[Theo.  5.8.3]{Bog}). Moreover, the operator
$\mbox{div}{}_{\gamma}$ is the adjoint of $-\nabla_{H}$ in $L^2(X,\gamma)$ in the sense that, if $f\in W^{1,2}(X;H)$ then 
\begin{eqnarray*}
\int_{X}\left\langle f,\nabla_{H}g\right\rangle {}_{H}\ d\gamma= -\int_{X}\mbox{div}{}_{\gamma}fg\ d\gamma
\end{eqnarray*}
 for every $g\in W^{1,2}(X,\gamma)$. 


A function $f:X\rightarrow \R$ is said to be $H$-Lipschitz continuous if there exists a positive constant $c$ such that for $\gamma$-a.e. $x\in X$ we have
\begin{align*}
|f(x+h)-f(x)|\leq c\|h\|_H, \qquad h\in H.    
\end{align*}
The constant $c$ is called the $H$-Lipschitz constant of $f$, and we denote by ${\rm Lip}_H(X)$ the space of $H$-Lipschitz continuous functions.

Let $f:\Omega\rightarrow X$ be a $H$-Lipschitz continuous function with $H$-Lipschitz constant $c>0$. Then, $f$ is G\^ateaux differentiable $\gamma$-a.e. in $X$, it is $H$-differentiable and it belongs to $W^{1,p}(X,\gamma)$. Moreover, $\nabla_{H}f$ is defined $\gamma$-a.e. in $X$, and 
\[
\|\nabla_{H}f\|_{H}\leq c, \qquad \gamma\textup{-a.e. in } \ X
\]
(see e.g. \cite[Theorem 5.11.2]{Bog}).
\subsection{The Hausdorff-Gauss spherical measure}

In the above setting, by following \cite{Fey}, it is possible to define a Borel measure $\mathcal{S}^{\infty-1}$ on $X$ which replace the $(d-1)$-Hausdorff measure in $\R^{d}$ in abstract Wiener spaces (hence, it can be seen as an area measure for $(\infty-1)$-hypersurfaces). The measure $\mathcal S^{\infty-1}$ is called the Hausdorff-Gauss measure or Feyel-de la Pradelle measure and we denote it by $\rho$. Let us briefly show the construction of $\rho$.

Let $F\subseteq R_{\gamma}(X^{*})$ be an $m$-dimensional subspace of $H$. We identify $F$ with $\R^{m}$ by choosing an orthonormal basis of $F$ in $H$, and by identifying it with the canonical basis of $\R^m$.
For every $m\in\N$, $\mathcal{S}^{m-1}$ denotes the spherical $(m-1)$-dimensional Hausdorff measure on space $F$, and for every $y\in X_{F^{\perp}}$ and every $B\in \mathcal B(X)$, we denote by $B_{y}$ the section
\[
B_{y}=\left\{ z\in F:\ y+z\in B\right\} 
\]
and the function 
\[
G_{m}(y)=(2\pi)^{-\frac{m}{2}}e^{-\frac{\left\Vert y\right\Vert _{H}^{2}}{2}}.
\]

The spherical $(\infty-1)$-dimensional Hausdorff-Gauss measure in $X$
with respect to $F$ is 
\[
{\mathcal{S}}_{F}^{\infty-1}(B)=\int_{X_{F^{\perp}}}\int_{B_{y}}G_{m}(z)\,d{\mathcal{S}}^{m-1}(z)\,d\gamma_{F}^{\perp}(y),\qquad  B\subseteq X.
\]
${\mathcal{S}}_{F}^{\infty-1}$ is a $\sigma$-additive Borel measure on $X$, and for every Borel set $B\in\mathcal B(X)$, the map $y\mapsto\int_{B_{y}}G_{m}\,d{\mathcal{S}}^{m-1}$
is $\gamma^{\perp}$-measurable in $F^{\perp}$. 
Since the measures ${\mathcal{S}}_{F}^{\infty-1}$ are monotone with respect to $F$, we set
\[
\rho(B)=\sup_{F\leq R_\gamma(X^{*})}\mathcal{S}_{F}^{\infty-1}(B)
\]
for every $B\in\mathcal B(X)$, where the supremum is meant as a supremum in a direct set. It turns out that $\rho$ is a Borel measure.

\subsection{Definition of the Sobolev spaces $W^{1,p}(O,\gamma)$ and $W^{1,p}_0(O,\gamma)$\label{sub:w1}}

Let $O\subseteq X$ be an open set. We denote by $\mathcal FC_b^\infty(O)$ the set of the restrictions to $O$ of elements of $\mathcal FC_b^\infty(X)$. The next Lemma is proved, for instance, in \cite[Lem. 2.1]{Add}.
\begin{lem}
\label{lem:Dirichlet_SObolev_O}
For every $p\in[1,+\infty)$, the
operator $\nabla_{H}:\F C_{b}^{\infty}(O)\rightarrow L^{p}(O,\gamma,H)$
is closable in $L^{p}(O,\gamma)$.
The same is true if we use $\mbox{Lip}(O)$
instead of $\F C_{b}^{\infty}(O)$, and the domains of the closures coincide. We still denote by $\nabla_H$ the closure of $\nabla_H$.
\end{lem}
\begin{proof}
The proof is the same of \cite[Lemma 2.1]{Add} for both  the space functions. The closures coincide because $\mathcal FC_b^\infty(O)
\subseteq {\rm Lip}(O)$, and every Lipschitz continuous function can be extended to $X$ by the McShane 
extension, and then approximated in $L^p$ by $\mathcal F C_b^\infty(X)$ functions.
\end{proof}
Actually, the proof in \cite{Add} uses spaces $\F C_{b}^{1}(O)$ and
$\mbox{Lip}_{b}(O)$, respectively, but the arguments are the same. From Lemma \ref{lem:Dirichlet_SObolev_O} we introduce the following spaces.
\begin{defn}
We denote by $W^{1,p}(O,\gamma)$ the domain of the closure of $\nabla_{H}$ in ${L^p(O,\gamma)}$. If endowed with the norm
\begin{align*}
\|f\|_{W^{1,p}(O,\gamma)}:=\left(\|f\|_{L^p(O,\gamma)}^p+\|\nabla_H f\|_{L^p(O,\gamma;H)}^p\right)^{1/p}, \quad f\in W^{1,p}(O,\gamma),
\end{align*}
the space $W^{1,p}(O,\gamma)$ is a Banach space. If $p=2$ then $W^{1,2}(O,\gamma)$ is a Hilbert space with inner product
\begin{align*}
\langle f,g\rangle_{W^{1,2}(O,\gamma)}=\langle f,g\rangle_{L^2(O,\gamma)}+\langle \nabla_Hf,\nabla_Hg\rangle_{L^2(O,\gamma;H)}, \quad f,g\in W^{1,2}(O,\gamma).
\end{align*}
\end{defn}

We now define the Sobolev spaces $W^{1,p}_0(O,\gamma)$.
\begin{defn}
\label{def:Sobolev_Dir} For every $p\in[1,+\infty)$, we denote by $W_{0}^{1,p}(O,\gamma)$ the closure of ${\rm Lip}_c(O)$ in $W^{1,p}(O,\gamma)$. \end{defn}

We want to prove that $W^{1,p}_0(O)$ actually coincides with the closure of different subspaces of $W^{1,p}(O,\gamma)$. To this aim, we introduce the following spaces of functions.
\begin{defn}
The space ${\rm Lip}_{c,H}(O)$ is the space of functions $f:O\rightarrow \R$, with support contained in an open set $A$ with positive distance from $O^c$, such that there exists a positive constant $\ell$ such that for $\gamma$-a.e. $x\in O$ we have
\begin{align}
\label{loc_H_lip_funct}
|f(x+h)-f(x)|\leq \ell   \|h\|_H, \quad \forall h\in H, \ x+h\in O.
\end{align}

\end{defn}

\begin{defn}
\label{def:C_b_H}With $\mathcal H^{1}(X)$ we denote the set of all continuous functions $f$ (not necessarily bounded) which are $H$-differentiable on $X$ and such that $\nabla_{H}f:X\rightarrow H$
is bounded and continuous.

$\mathcal H_{0}^{1}(O)$ is the subset of $\mathcal H^1(X)$ of functions $f$ whose support has positive distance from $O^c$. 
\end{defn}

The following result shows that $\mathcal H_0^1(O)$ is not empty. To prove this fact, we introduce the Ornstein-Uhlenbeck semigroup $(T_t)_{t\geq0}$ on $X$, characterized by the Mehler formula
\begin{align}
\label{OU_sempigroup}
T_tf(x)=\int_Xf\left(e^{-t}x+\sqrt{1-e^{-2t}}y\right)\gamma(dy),
\qquad f\in C_b(X), \ x\in X, \ t\geq0,
\end{align}
which extends to a bounded strongly continuous semigroup on $L^p(X,\gamma)$ for every $p\in[1,+\infty)$, which we again denote by $(T_t)_{t\geq0}$. We recall that for every $f\in C_b(X)$, we have $T_tf$ is $H$-differentiable and
\begin{align*}
\langle \nabla_HT_tf(x),h\rangle_H
= \frac{e^{-t}}{\sqrt{1-e^{2t}}}\int_Xf\left(e^{-t}x+\sqrt{1-e^{-2t}}y\right)\widehat h(y)\gamma(dy),
\qquad  \ x\in X, \ t\geq0.
\end{align*}

\begin{lem}
\label{lem:spazio_funz_test_prop}
$\mathcal H_{0}^{1}(O)$ is not empty.
\end{lem}

\begin{proof}
Let us fix a bounded closed set $B\subseteq X$ and $\varepsilon>0$. From \cite[Lemma 2.5]{Su98} and its proof, we infer that there exists a Lipschitz continuous function $f$ with Lipschitz constant $2\varepsilon^{-1}$, $t>0$ and a smooth function $\Phi\in C_b^\infty(\R)$ with $0\leq \Phi\leq 1$, $\Phi=0$ on $(-\infty,1/3)$ and $\Phi=1$ on $(2/3,+\infty)$, such that the function $F_{B,\varepsilon}=\Phi(T_tf)$ equals $1$ on $B$ and $F_{B,\varepsilon}=0$ on $X\setminus\{x\in X:d(x,B)>\varepsilon\}$.



Hence, for every bounded open set $ A\subseteq O$ with positive distance $d$ from $O^c$, with the choice $B=\overline A$ and $\varepsilon<d$, the function $F=F_{B,\varepsilon}$ has $\nabla_HF$ everywhere defined and bounded by \cite[Theorem 5.11.2]{Bog}. $F$ belongs to $\mathcal H_0^1(O)$, providing that we prove that $\nabla_HF$ is continuous. To this aim, for every $x,y\in X$ we have
\begin{align*}
& \|\nabla _HF(x)-\nabla_HF(y)\|_H^2 \\
\leq  & \frac{e^{-2t}\|\Phi\|_{C^1_b(\R)}^2}{1-e^{-2t}}\sum_{n=1}^\infty\left(\int_X\left(f(e^{-t}x+\sqrt{1-e^{-2t}}z)-f(e^{-t}y+\sqrt{1-e^{-2t}}z)\right)\hat h_n(z)\gamma(dz)\right)^2 \\
\leq &  \frac{e^{-2t}\|\Phi\|_{C^1_b(\R)}^2}{1-e^{-2t}} \|f(e^{-t}x+\sqrt{1-e^{-2t}}\cdot)-f(e^{-t}y+\sqrt{1-e^{-2t}}\cdot)\|_{L^2(X,\gamma)}^2 \\
\leq &  \frac{4e^{-4t}\|\Phi\|_{C^1_b(\R)}^2}{\varepsilon^2(1-e^{-2t})}\|x-y\|_X^2,
\end{align*}
where the second inequality is a consequence of \eqref{sum}, and this gives the continuity of $\nabla_HF$. Further, $F$ is Lipschitz continuous due to the Lipschitz continuity of $f$, the definition of $T_t$ and the smoothness of $\Phi$. 
\end{proof}



From the definition of $\mathcal H^1_0(O)$, it follows that $\mathcal H^1_0(O)\subseteq {\rm Lip}_{c,H}(O)$.
Further, for every $f\in{\rm Lip}_{c,H}(O)$, its trivial extension $\overline f$ belongs to $W^{1,p}(X,\gamma)$ for every $p\in(1,+\infty)$, and so ${\rm Lip}_{c,H}(O)\subseteq W^{1,p}(O,\gamma)$ for every $p\in(1,+\infty)$ (see \cite[Theorem 5.11.2]{Bog}).

Clearly, also ${\rm Lip}_c(O)\subseteq {\rm Lip}_{c,H}(O)$, and so \begin{align*}
W^{1,p}_0(O,\gamma)\subseteq\overline{{\rm Lip}_{c,H}(O)}^{W^{1,p}(O,\gamma)}, \qquad \overline{{\mathcal H}_{0}^1(O)}^{W^{1,p}(O,\gamma)}\subseteq\overline{{\rm Lip}_{c,H}(O)}^{W^{1,p}(O,\gamma)}.    
\end{align*}
We prove that the above inclusions are indeed equalities.
\begin{lem}
\label{lem:dens_H_funct_sob_dir}
We have
\begin{align*}
W^{1,p}_0(O,\gamma)=\overline{{\mathcal H}_{0}^1(O)}^{W^{1,p}(O,\gamma)}=
\overline{{\rm Lip}_{c,H}(O)}^{W^{1,p}(O,\gamma)},    
\end{align*}
 The closure of $\mathcal H_0^1(O)$ in $W^{1,p}(O,\gamma)$ coincides with $W^{1,p}_0(O,\gamma)$ 
for every $p\in(1,+\infty)$.
\end{lem}
\begin{proof}
Let us fix $p\in(1,+\infty)$. To prove the statement, we show that $\overline{{\mathcal H}_{0}^1(O)}^{W^{1,p}(O,\gamma)}\subseteq W^{1,p}_0(O,\gamma)$ and that $\overline{{\rm Lip}_{c,H}(O)}^{W^{1,p}(O,\gamma)}\subseteq \overline{{\mathcal H}_{0}^1(O)}^{W^{1,p}(O,\gamma)}$. Without loss of generality, we assume that the support of the considered functions is bounded.

\vspace{2mm}

Let $g\in \mathcal H^1_0(O)$. Its trivial extension $\overline g$ belongs to $ W^{1,p}(X,\gamma)$ and so there exists a sequence $(g_n)_{n\in\N}\subseteq \mathcal{FC}_b^\infty(X)$ which converges to $\overline g$ in $W^{1,p}(X,\gamma)$ as $n\rightarrow +\infty$. Let $A\subseteq O$ be a bounded open set with positive distance from $O^c$ such that ${\rm supp}(g)\subseteq A$, and let $F$ be the function defined in Lemma \ref{lem:spazio_funz_test_prop}. $F$ is Lipschitz continuous: indeed, for every $x,y\in X$ we have
\begin{align*}
|F(x)-F(y)|
\leq & \|\Phi'\|_\infty|T_tf(x)-T_tf(y)| \\
\leq & \|\Phi'\|_\infty\int_X\left|f(e^{-t}x+\sqrt{1-e^{2t}}z)-f(e^{-t}y+\sqrt{1-e^{2t}}z)\right|\gamma(dz) \\
\leq & \frac{2}{\varepsilon}e^{-t}\|\Phi'\|_\infty|x-y|_X,
\end{align*}
where we have used the fact that $f$ is a $\frac{2}{\varepsilon}$-Lipschitz continuous function.
Then, the sequence $(F g_n)_{n\in\N}\subseteq \mathcal {\rm Lip}_c(O)$ and it converges to $g$ in $W^{1,p}(O,\gamma)$. 

This implies that $\overline{{\mathcal H}_{0}^1(O)}^{W^{1,p}(O,\gamma)}\subseteq W^{1,p}_0(O,\gamma)$.


\vspace{2mm}
Let $g\in {\rm Lip}_c(O)$. Its trivial extension $\overline g$ belongs to $ W^{1,p}(X,\gamma)$. Hence, there exists a sequence $(g_n)_{n\in\N}\subseteq \mathcal {FC}_b^\infty(X)$ such that $g_n\rightarrow \overline g$ in $W^{1,p}(X,\gamma)$ as $n\rightarrow+\infty$. Let $A\subseteq O$ be a bounded open set with positive distance from $O^c$ such that ${\rm supp}(g)\subseteq A$, and let $F$ be the function defined in Lemma \ref{lem:spazio_funz_test_prop}.  The sequence $(F g_n)_{n\in\N}$ converges to $g$ in $W^{1,p}(O,\gamma)$, with $F\in \mathcal H_0^1(O)$ and $g_n\in \mathcal {FC}_b^\infty(X)$, which give $F g_n \in \mathcal H_0^1(O)$ for every $n\in\N$.

This gives $\overline{{\rm Lip}_{c,H}(O)}^{W^{1,p}(O)}\subseteq \overline{{\mathcal H}_{0}^1(O)}^{W^{1,p}(O)}$.
\end{proof}


By the operator theory, there exists a unique unbounded operator $L_O$, with dense domain in $W_{0}^{1,2}(O,\gamma)$, such that, for every $f\in D(L_O)$ and $g\in W_{0}^{1,2}(O,\gamma)$, we get
\[
\int_{O}L_Of\cdot g\ d\gamma=-\int_{O}\left\langle \nabla_{H}f,\nabla_{H}g\right\rangle _{H}\ d\gamma.
\]

\begin{defn}
The operator $L_O:D(L_O)\rightarrow L^2(O,\gamma)$ is called Ornstein-Uhlenbeck operator on $O$ with homogeneous Dirichlet boundary conditions. 

When $O=X$, we denote $L_X$ by $L$, and it is the infinitesimal generator of the Ornstein-Uhlenbeck semigroup $(T_{t})_{t\geq0}$.
\end{defn}


\section{Traces in regular sets}
In the following, for every $\delta>0$ we denote by $I_{\delta}$ the real interval $(-\delta,\delta)\subseteq\R$.


Inspired by \cite[Hypothesis 3.1]{Cel}, we state our assumptions on $O$.

\begin{hyp}
\label{claim:regularity} 
Let $G:X\rightarrow \R$ and $\delta>0$ satisfy:
\begin{enumerate}[\rm(i)]
\item
$G$ is a continuous function which belongs to ${\rm Lip}_{H}(X)$;
\item $G\in W^{2,p}(X,\gamma)$ for some $p>1$ and ${\rm esssup}_{X}\|D^2_HG\|_{\mathcal L_2(H)}<+\infty$;
\item $\|\nabla_{H}G\|_{H}^{-1}\in L^{\infty}(X)$;
\item $LG\in L^\infty(G^{-1}(I_{\delta}))$.
\end{enumerate}
Hereafter, we set $O:=G^{-1}((-\infty,0))$ and we assume that $O$ and $\partial O$ are not the empty set. 
\end{hyp}
\begin{rem}
\label{rem:G_boundary}
Let us comment the above assumptions.
\begin{enumerate}
\item[i)] $O$ is an open set and $\partial O=G^{-1}(\{0\})$.
Hence, $\gamma(O)>0$ since every open set has positive measure by an immediate consequence of \cite[Prop. 2.4.10]{Bog}.
\item[ii)] From Hypothesis \ref{claim:regularity}$(ii)$ we infer that $G\in W^{2,q}(X,\gamma)$
for all $q>1$, and so \cite[Hypothesis 3.1]{Cel} is fulfilled.
\item[iii)] By the points (ii) and (iii) of the Hypotheses \ref{claim:regularity} it
follows that 
\begin{align*}
\frac{\nabla_{H}G}{\|\nabla_{H}G\|_{H}}\in W^{1,2}(X;H).    
\end{align*}
By adding the point (iv) we have also that ${\rm div}_{\gamma}\left(\frac{\nabla_{H}G}{\|\nabla_{H}G\|_{H}}\right)\in L^{\infty}(G^{-1}(\I))$, since
\begin{align*}
\mbox{\rm div}_{\gamma}\left(\frac{\nabla_{H}G}{\|\nabla_{H}G\|_{H}}\right)=\frac{ LG}{\|\nabla_{H}G\|_{H}}-\frac{\left\langle D_{H}^{2}G(\nabla_{H}G),\nabla_{H}G\right\rangle _{H}}{\|\nabla_{H}G\|_{H}^{3}}    \end{align*}

\item[iv)] The Hypothesis (iii) is very restrictive, for example it
is not satisfied by $\|\cdot\|_{X}^{2}$ when $X$ is a Hilbert space, which would allow to consider balls. 

\item[v)] For $-\delta<\varepsilon<\delta$ the assumption remains
true if we replace $G$ with $G+\varepsilon$ or with $-G+\varepsilon$, with the value $\delta$ replaced by $\delta'=\delta-|\varepsilon|$.

\item[vi)] From Hypotheses \ref{claim:regularity} it follows that $LG$, and so ${\rm div}_\gamma(\nabla_HG/\|\nabla_HG\|_H)$, belongs to $ L^{p}(X,\gamma)$
for every $p\in(1,+\infty)$.

\end{enumerate}
\end{rem}

In the sequel, we will need the Sobolev regularity of the modulus of elements of $W^{1,p}(X)$, which is proved in the following lemma.

\begin{lem}
\label{prop:mod_Sobolev_reg}
Let $u\in W^{1,p}(X,\gamma)$ with $p>1$. Then, for every $q\in(1,p)$, the function $|u|^{q}$ belongs to $W^{1,p/q}(X,\gamma)$, and $\nabla_H|u|^q=q\ \!{\rm sgn}(u)|u|^{q-1}\nabla_Hu$.
\end{lem}
\begin{proof}
The classical method consists in introducing the function $\eta_n(\xi):=\left(\xi^2+\frac1n\right)^{q/2}$ and approximating $|u|^q$ by means of the sequence $(\eta_n\circ u_n)\subseteq \mathcal FC_b^\infty(X)$ in $W^{1,p/q}(O,\gamma)$, where $(u_n)\subseteq \mathcal FC_b^\infty(X)$ is a sequence which converges to $u$ in $W^{1,p}(O,\gamma)$. However, we provide a different proof.

At first, we notice that, for every $q\geq 1$ and every $a,b\in\R$, we have
\begin{align}
\label{dis_pot_q}
\left|\text{sgn}(a)\left|a\right|^{q}-\text{sgn}(b)\left|b\right|^{q}\right|\leq q\left|a-b\right|\left|\left|a\right|^{q-1}+\left|b\right|^{q-1}\right|.    
\end{align}
Let $(u_{n})_{n\in \N}\subseteq \mathcal FC_b^\infty(X)$ be a sequence which converges to $u$ in $W^{1,p}(X,\gamma)$. Without loss of generality, we may suppose that $(u_n)_{n\in\N}$ pointwise converges to $u$. We split the proof into three steps.

\emph{Step 1.} Here we prove that $\left|u_{n}\right|^{q}$ converges to $\left|u\right|^{q}$
in $L^{p/q}(X,\gamma)$. Since $q>1$, from \eqref{dis_pot_q} it follows that
\begin{align}
\notag
\left\Vert \left|u_{n}\right|^{q}-\left|u\right|^{q}\right\Vert _{L^{p/q}(X,\gamma)}
\leq & q\left\Vert \left(\left|u_{n}\right|-\left|u\right|\right)\left(\left|u_{n}\right|^{q-1}+\left|u\right|^{q-1}\right)\right\Vert _{L^{p/q}(X,\gamma)} \\
\leq & q\left\Vert u_{n}-u\right\Vert _{L^{p}(X,\gamma)}\left\Vert \left|u_{n}\right|^{q-1}+\left|u\right|^{q-1}\right\Vert _{L^{p/(q-1)}(X,\gamma)},
\label{stima_1_modulo}
\end{align}
where we have used the H\"older inequality with $q$ and $q/(q-1)$. The first factor in the right-hand side of \eqref{stima_1_modulo} converges to $0$ as $n\rightarrow +\infty$, while the second is bounded, uniformly with respect to $n\in\N$, since
\begin{align*}
\left\Vert \left|u_{n}\right|^{q-1}+\left|u\right|^{q-1}\right\Vert _{L^{p/(q-1)}(X,\gamma)} 
\leq & \left\Vert \left|u_{n}\right|^{q-1}\right\Vert _{L^{p/(q-1)}(X,\gamma)}+\left\Vert \left|u\right|^{q-1}\right\Vert _{L^{p/(q-1)}(X,\gamma)} \\ 
\leq & \left\Vert u_{n}\right\Vert _{L^{p}(X,\gamma)}^{q-1}+\left\Vert u\right\Vert _{L^{p}(X,\gamma)}^{q-1},
\end{align*}
which is uniformly bounded with respect to $n\in\N$.

\emph{Step 2. }We want to show that $\text{sgn}(u_{n})|u_{n}|^{q-1}$
converges to $\text{sgn}(u)|u|^{q-1}$ in $L^{p/(q-1)}(X,\gamma)$.\\
Let us suppose $q\geq2$, hence $q-1\geq1$, and from \eqref{dis_pot_q} we infer that
\begin{align*}
& \int_{X}\left|\text{sgn}(u_{n})\left|u_{n}\right|^{q-1}-\text{sgn}(u)\left|u\right|^{q-1}\right|^{p/(q-1)}d\gamma \\
\leq & (q-1)^{p/(q-1)}\int_{X}\left|u_{n}-u\right|^{p/(q-1)}\left|\left|u_{n}\right|^{q-2}+\left|u\right|^{q-2}\right|^{p/(q-1)}d\gamma \\
\leq & (q-1)^{p/(q-1)}\left(\int_{X}\left|u_{n}-u\right|^{p}d\gamma\right)^{\frac{1}{q-1}} \\
& \times\left(\int_{X}\left|\left|u_{n}\right|^{q-2}+\left|u\right|^{q-2}\right|^{p/(q-2)}d\gamma\right)^{(q-2)/(q-1)} \\
=: & (q-1)^{p/(q-1)}A_{n}B_{n}.
\end{align*}
$A_{n}$ converges to $0$ as $n\rightarrow+\infty$, and 
\begin{align*}
B_{n}^{(q-1)/p}\leq\left\Vert \left|u_{n}\right|^{q-2}+\left|u\right|^{q-2}\right\Vert _{L^{p/(q-2)}(X,\gamma)}
\leq& \left\Vert \left|u_{n}\right|^{q-2}\right\Vert _{L^{p/(q-2)}(X,\gamma)}+\left\Vert \left|u\right|^{q-2}\right\Vert _{L^{p/(q-2)}(X,\gamma)} \\
\leq & \left\Vert u_{n}\right\Vert _{L^{p}(X,\gamma)}^{q-2}+\left\Vert u\right\Vert _{L^{p(q-2)}(X,\gamma)}^{q-2},
\end{align*}
which is uniformly bounded with respect to $n\in\N$.

Instead, if $q\in[1,2)$,
\begin{align*}
& \left\Vert \text{sgn}(u_{n})\left|u_{n}\right|^{q-1}-\text{sgn}(u)\left|u\right|^{q-1}\right\Vert _{L^{p/(q-1)}(X,\gamma)} \\
\leq & \left\Vert \text{sgn}(u_{n})\left|u_{n}\right|^{q-1}-\text{sgn}(u_{n})\left|u\right|^{q-1}\right\Vert _{L^{p/(q-1)}(X,\gamma)} \\
& +\left\Vert \text{sgn}(u_{n})\left|u\right|^{q-1}-\text{sgn}(u)\left|u\right|^{q-1}\right\Vert _{L^{p/(q-1)}(X,\gamma)} \\
\leq &\left\Vert \left|u_{n}\right|^{q-1}-\left|u\right|^{q-1}\right\Vert _{L^{p/(q-1)}(X,\gamma)} \\
& +\left(\int_{X}\left(\left(\text{sgn}(u_{n})-\text{sgn}(u)\right)\left|u\right|^{q-1}\right)^{p/(q-1)}d\gamma\right)^{(q-1)/p}
=:P_{n}+Q_{n}.
\end{align*}
 Since $t\mapsto t^{q-1}$ is concave, we get 
\[
P_{n}\leq\left\Vert \left|\left|u_{n}\right|-\left|u\right|\right|^{q-1}\right\Vert _{L^{p/(q-1)}(X,\gamma)}=\left\Vert u_{n}-u\right\Vert _{L^{p}(X,\gamma)}^{q-1}
\]
which converges to $0$ as $n\rightarrow+\infty$, while, if we set $U:=\{x\in X|u(x)\neq0\}$, then
\[
Q_{n}=\int_{U}\left(\text{sgn}(u_{n})-\text{sgn}(u)\right)\left|u\right|^{p}d\gamma,
\]
and it converges to $0$ by the dominated convergence, since $(u_n)_{n\in\N}$ pointwise converges to $u$.

\emph{Step 3.} By approximation, it is possible to show that $\nabla_H |u_{n}|^q=q\ \!\text{sgn}(u_{n})|u_{n}|^{q-1}\nabla_{H}u_{n}$ for every $n\in\N$. In this last step we prove that $\nabla_H |u_{n}|^q$
converges to $q\ \!\text{sgn}(u)|u|^{q-1}\nabla_{H}u$ in $L^{p/q}(X,\gamma,H)$ as $n\rightarrow+\infty$. We have
\begin{align*}
& \left\Vert q\ \!\text{sgn}(u_{n})|u_{n}|^{q-1}\nabla_{H}u_{n}-q\ \!\text{sgn}(u)|u|^{q-1}\nabla_{H}u\right\Vert _{L^{p/q}(X,\gamma)} \\
\leq & 
\left\Vert q\ \!\text{sgn}(u_{n})|u_{n}|^{q-1}\nabla_{H}u_{n}-q\ \!\text{sgn}(u)|u|^{q-1}\nabla_{H}u_{n}\right\Vert _{L^{p/q}(X,\gamma,H)} \\
&+\left\Vert q\ \!\text{sgn}(u)|u|^{q-1}\nabla_{H}u_{n}-q\ \!\text{sgn}(u)|u|^{q-1}\nabla_{H}u\right\Vert _{L^{p/q}(X,\gamma,H)}=:R_{n}+S_{n}.
\end{align*}
As far as $R_n$ is concerned, by applying the H\"older inequality with $q$ and $q/(q-1)$ we get
\begin{align*}
R_{n}\leq q\left\Vert \text{sgn}(u_{n})|u_{n}|^{q-1}-\text{sgn}(u_{n})|u_{n}|^{q-1}\right\Vert _{L^{p/(q-1)}(X,\gamma)}^{(q-1)/q}\cdot\left\Vert \nabla_{H}u_n\right\Vert _{L^{p}(X,H,\gamma)}^{\frac{1}{q}}.
\end{align*}
The first factor converges to $0$ as $n\rightarrow +\infty$ from Step 2 and the second one is uniformly bounded with respect to $n\in\N$, while 
\begin{align*}
S_{n}
\leq & q\left\Vert |u|^{q-1}\right\Vert _{L^{p/(q-1)}(X,\gamma)}^{(q-1)/q}\cdot\left\Vert \nabla_{H}u_{n}-\nabla_{H}u\right\Vert _{L^{p}(X,H,\gamma)}^{\frac{1}{q}} \\
= & \left\Vert q^{1/(q-1)}|u|\right\Vert _{L^{p}(X,\gamma)}^{1/q}\cdot\left\Vert \nabla_{H}u_{n}-\nabla_{H}u\right\Vert _{L^{p}(X,H,\gamma)}^{\frac{1}{q}},
\end{align*}
and the last term converges to $0$ as $n\rightarrow+\infty$, where again we have used the H\"older inequality with $q$ and $a/(a-1)$.

The proof is concluded.

\end{proof}

From \cite{Cel}, under Hypotheses \ref{claim:regularity}
(i)-(iii), for every $t\in(-\delta,\delta)$ and every $q<p$ it is defined the {\em trace operator} 
\[
\Tr_{t}:W^{1,p}(G^{-1}((-\infty,t)),\gamma)\rightarrow L^{q}(G^{-1}(\{t\}),\rho).
\]
If $f\in W^{1,p}(G^{-1}((-\infty,t)),\gamma)$ is the restriction of a continuous function on $X$, then 
\[
\Tr_{t}f=f_{|G^{-1}(\{t\})}.
\]

The following three results are proved in \cite{Cel}.
\begin{lem}{\cite[Cor. 3.2]{Cel}}
\label{lem:Density}
Assume Hypotheses \ref{claim:regularity} (i)-(iii), let $\delta_{0}>0$ and $O_{\delta_{0}}:=G^{-1}(I_{\delta_{0}})$. Then, for every $f\in {\rm Lip}(X)\subseteq L^{1}(O_{\delta_0},\gamma)$, the function 
\[
q_{f}(\xi):=\int_{G^{-1}(\{\xi\})}\frac{f}{\|\nabla_{H}G\|_{H}}\ d\rho, \quad -\delta_0<\xi<\delta_0,
\]
belongs to $L^{1}(I_{\delta_{0}},\mbox{\ensuremath{\mathscr{L}}}^{1})$ (with
$\mbox{\ensuremath{\mathscr{L}}}^{1}$ being the $1$-dimensional Lebesgue measure). Moreover, $q_{f}$ is a density of the measure $f\gamma\circ G^{-1}$ with respect to $\mbox{\ensuremath{\mathscr{L}}}^{1}$, and
\begin{align*}
\|q_f\|_{L^1(-\delta_0,\delta_0)}\leq \|f\|_{L^1(O_{\delta_0},\gamma)}.    
\end{align*}
\end{lem}
\medskip{}

\begin{lem}{\cite[Prop. 4.10]{Cel}}
\label{lem:Trace_0-1}Under Hypotheses \ref{claim:regularity}, for every $p>1$, every $t\in\I$ and every $f\in W^{1,p}(G^{-1}((-\infty,t)),\gamma)$, $\Tr_{t}f\equiv0$ if and only if the trivial extension $\overline f$ of $f$ out
of $G^{-1}((-\infty,t))$ belongs $W^{1,p}(X,\gamma)$.
\end{lem}
\medskip{}

\begin{lem}{\cite[Prop. 4.1]{Cel}}
\label{lem:Celada_1} Under Hypotheses \ref{claim:regularity}, for every $p>1$, every $q\in[1,p)$ and every $t\in\I$, if $f\in W^{1,p}(X,\gamma)$ then
\begin{align}
& \int_{G^{-1}(\{t\})}|\Tr_{t}f|^{q} d\rho\notag \\
 =& q\int_{G^{-1}((-\infty,t))}|f|^{q-2}f\frac{\left\langle \nabla_{H}f,\nabla_{H}G\right\rangle _{H}}{\|\nabla_HG\|_H} d\gamma+\int_{G^{-1}((-\infty,t))}{\rm div}_\gamma\left(\frac{\nabla_HG}{\|\nabla_HG\|_H}\right)|f|^{q}\ d\gamma \notag \\
= & q\int_{G^{-1}((t,+\infty))}|f|^{q-2}f\frac{\left\langle \nabla_{H}f,\nabla_{H}G\right\rangle _{H}}{\|\nabla_HG\|_H} d\gamma+\int_{G^{-1}((t,+\infty))}{\rm div}_\gamma\left(\frac{\nabla_HG}{\|\nabla_HG\|_H}\right)|f|^{q}\ d\gamma. 
\label{eq:Celada_1}
\end{align}
\end{lem}

\begin{rem}
\label{rem:null-boundary}An easy consequence of Lemma \ref{lem:Density}
is that $\gamma(G^{-1}(\{t\}))=0$ for every $t\in\I$. Further, if we take $f=1$ in 
\eqref{eq:Celada_1}, we infer that
\begin{align}
\notag
\rho(G^{-1}(\{t\}))
= & \int_{G^{-1}(\{t\})}d\rho \\
\leq & \left\|{\rm div}_\gamma\left(\frac{\nabla_HG}{\|\nabla_HG\|_H}\right)\right\|_{L^{1}(O,\gamma)}+\left\|{\rm div}_\gamma\left(\frac{\nabla_HG}{\|\nabla_HG\|_H}\right)\right\|_{L^{\infty}(G^{-1}(I_\delta))}<+\infty.
\label{stima_rho_livelli_G}
\end{align}
\end{rem}

\section{Equivalent definitions of $W^{1,p}_0(O,\gamma)$}
We set $A_{t}:=G^{-1}((t,\delta))$, and we prove the following two intermediate results.

\begin{lem}
\label{lem:ineq_trace} 
Let Hypotheses \ref{claim:regularity} be satisfied. Then, for every $q\in[1,+\infty)$, there exists $C>0$ such that for every $t\in(-\delta,0)$
and $f\in W^{1,q}(X,\gamma)$ we have
\begin{equation}
\|\Tr_{t}f\|_{L^{q}(G^{-1}(\{t\}),\rho)}^{q}\leq C(\|f\|_{L^{q}(A_{t},\gamma)}^{q-1}\|\nabla_{H}f\|_{L^{q}(A_{t},\gamma,H)}+\|f\|_{L^{q}(A_{t},\gamma)}^{q}).\label{eq:ineq_trace}
\end{equation}
\end{lem}
\begin{proof}
By density, it is enough to consider Lipschitz continuous functions $f$.

Arguing as in the proof of \cite[Prop. 4.1]{Cel}, we introduce a function $\theta\in C_{b}^{\infty}(\R)$ such that $\theta=1$ in $(-\infty,0]$, $\theta=0$ in $[\delta,+\infty)$ and $\theta(x)\in[0,1]$ for every $x\in\R$. We define the function $\psi:=f\cdot(\theta\circ G)$. The function $\psi$ belongs to $ W^{1,s}(X,\gamma)$ for every $s\in[1,+\infty)$ (because
$f$ is Lipschitz continuous and $G\in \mathcal H^1(X)$) with $\nabla_H\psi=(\theta\circ G)\nabla_Hf+f(\theta'\circ G)\nabla_HG$. From its definition, $\psi =0$ on $G^{-1}((\delta,+\infty))$, $|\psi|\leq|f|$ on $X$, and
\[
\|\psi\|_{L^{s}(G^{-1}((t,+\infty)),\gamma)}=\|\psi\|_{L^{s}(A_{t},\gamma)},
\qquad 
\|\nabla_{H}\psi\|_{L^{s}(G^{-1}((t,+\infty));\gamma,H)}=\|\nabla_{H}\psi\|_{L^{s}(A_{t},\gamma,H)}.
\]
Finally, $\psi_{|G^{-1}(\{t\})}\equiv f$. Hence, we can apply Lemma
\ref{lem:Celada_1}, which gives
\begin{align*}
& \int_{G^{-1}(\{t\})}|\Tr_{t}f|^{q}d\rho
=  \int_{G^{-1}(\{t\})}|\Tr_{t}\psi|^{q} d\rho \\
=& q\int_{A_{t}}|\psi|^{q-2}\psi\frac{\left\langle \nabla_{H}\psi,\nabla_{H}G\right\rangle _{H}}{\|\nabla_HG\|_H}\ d\gamma+\int_{A_{t}}{\rm div}_\gamma\left(\frac{\nabla_HG}{\|\nabla_HG\|_H}\right)|\psi|^{q}\ d\gamma \\
\leq & 
q\int_{A_{t}}|\psi|^{q-1}\|\nabla_{H}\psi\|_{H}\ d\gamma+\int_{A_{t}}{\rm div}_\gamma\left(\frac{\nabla_HG}{\|\nabla_HG\|_{H}}\right)|\psi|^{q}\ d\gamma .
\end{align*}
Recalling that ${\rm div}_\gamma(\nabla_HG/\|\nabla_HG\|_H)\in L^{\infty}(A_{t},\gamma)$ (see Remark \ref{rem:G_boundary}$(iii)$), we infer that
\begin{align*}
& \int_{G^{-1}(\{t\})}|\Tr_{t}f|^{q}d\rho \\
\leq & q\|\psi\|_{L^{q}(A_{t},\gamma)}^{q-1}\|\nabla_{H}\psi\|_{L^{q}(A_{t},\gamma,H)}+ \left\|{\rm div}_\gamma\left(\frac{\nabla_HG}{\|\nabla_HG\|_H}\right)\right\|_{L^{\infty}(A_{t},\gamma)}\|\psi\|_{L^{q}(A_{t},\gamma)}^{q} \\
\leq & C_1\|f\|_{L^{q}(A_{t},\gamma)}^{q-1}(\|\nabla_{H}f\|_{L^{q}(A_{t},\gamma,H)}+\|f\|_{L^{q}(A_{t},\gamma)})+C_2\|f\|_{L^{q}(A_{t},\gamma)}^{q} \\
\leq & (C_1+C_2)(\|f\|_{L^q(A_t,\gamma)}^q+\|\nabla_Hf\|_{L^q(A_t,\gamma;H)}),
\end{align*}
where $C_1=q(1+\|\theta'\|_\infty\|\nabla_HG\|_{L^\infty(X;H)})$ and $C_2=\left\|{\rm div}_\gamma\left(\frac{\nabla_HG}{\|\nabla_HG\|_H}\right)\right\|_{L^{\infty}(A_{t},\gamma)}$. The proof is now complete.
\end{proof}

\begin{lem}
\label{lem:pseudo_Poinc}Let Hypotheses \ref{claim:regularity} be satisfied. Then, for every $p>1$ there exists $C_{0}>0$ and $\delta_{0}\in(0,\delta)$ such that, for every $f\in W^{1,p}(O,\gamma)$ with $\Tr_{0}f\equiv0$ on $G^{-1}(\{0\})$ and every $t\in(-\delta_{0},0)$, we have 
\[
\|f\|_{L^p(G^{-1}((t,0)),\gamma)}\leq 2C_1|t|\|\nabla_{H}f\|_{L^{p}(G^{-1}((t,0)),\gamma;H)}.
\]
\end{lem}
\begin{proof}
By density, it is enough to consider Lipschitz continuous functions $f$.

Let us assume that $\|f\|_{L^p(G^{-1}((t,0),\gamma)}\neq 0$. From Lemma \ref{lem:Trace_0-1}, the trivial extension $\overline f$ of $f$ belongs to $W^{1,p}(X,\gamma)$. From Lemma \ref{prop:mod_Sobolev_reg} it follows that $|\overline f|^q\in L^{p/q}(X,\gamma)$ for every $q\in(1,p)$, and by applying Lemma \ref{lem:ineq_trace} to the function $|\overline f|^q$, for every $s\in(-\delta,0)$ we get
\begin{align*}
\|\Tr_{s}| f|^q\|_{L^{p/q}(G^{-1}(\{s\}),\rho)}^{p/q}
= & \|\Tr_{s}|\overline f|^q\|_{L^{p/q}(G^{-1}(\{s\}),\rho)}^{p/q} \\
\leq & C(\|f\|_{L^{p}(G^{-1}((s,0)),\gamma)}^{p-q}\||f|^{q-1}\nabla_{H}f\|_{L^{p/q}(G^{-1}((s,0)),\gamma; H)} \\
& +\|f\|_{L^{p}(G^{-1}((s,0)),\gamma)}^{q}) \\
\leq & C(\|f\|_{L^{p}(G^{-1}((s,0)),\gamma)}^{p-q}\|f\|_{L^p(G^{-1}((s,0),\gamma)}^{q-1}\|\nabla_{H}f\|_{L^{p}(G^{-1}((s,0)),\gamma; H)} \\
& +\|f\|_{L^{p}(G^{-1}((s,0)),\gamma)}^{q}) \\
\leq & C(\|f\|_{L^{p}(G^{-1}((s,0)),\gamma)}^{p-1}\|\nabla_{H}f\|_{L^{p}(G^{-1}((s,0)),\gamma; H)} +\|f\|_{L^{p}(G^{-1}((s,0)),\gamma)}^{p}),
\end{align*}
where $C=\frac{p}{q}(1+\|\theta'\|_\infty\|\nabla_HG\|_{L^\infty(X;H)})+\left\|{\rm div}_\gamma\left(\frac{\nabla_HG}{\|\nabla_HG\|_H}\right)\right\|_{L^{\infty}(A_{t},\gamma)}$ and we have used the fact that $\bar{f}\equiv0$ in $G^{-1}((0,+\infty))$ (hence the norm of $f$ in $G^{-1}((s,+\infty))$ coincides with that in $G^{-1}((s,0))$). In particular, it follows that
\begin{align}
\label{stima_traccia_q_int}
\|\Tr_{s}| f|^q\|_{L^{p/q}(G^{-1}(\{s\}),\rho)}^{p/q}
\leq \widetilde C(\|f\|_{L^{p}(G^{-1}((s,0)),\gamma)}^{p-1}\|\nabla_{H}f\|_{L^{p}(G^{-1}((s,0)),\gamma; H)} +\|f\|_{L^{p}(G^{-1}((s,0)),\gamma)}^{p}),
\end{align}
where $\widetilde C := p(1+\|\theta'\|_\infty\|\nabla_HG\|_{L^\infty(X;H)})+\left\|{\rm div}_\gamma\left(\frac{\nabla_HG}{\|\nabla_HG\|_H}\right)\right\|_{L^{\infty}(I_{\delta},\gamma)}$, for every $q\in(1,p)$ and every $s\in(-\delta,0)$.

Let $q\in(1,p)$. From Lemma \ref{lem:Density}, for every $t\in(-\delta,0)$ we have
\begin{align}
\int_{G^{-1}((t,0))}|f|^qd\gamma
= & \int_t^0\int_{G^{-1}(\{\xi\})}\frac{{\rm Tr}_{\xi}|f|^q}{\|\nabla_HG\|_H}d\rho d\xi \notag \\
\leq & |t|\left\|\frac{1}{\|\nabla_HG\|_H}\right\|_{L^\infty(G^{-1}((-\delta,0))}\sup_{s\in(t,0)}\|{\rm Tr}_s|f|^q\|_{L^1(G^{-1}(\{s\}),\rho)} \notag\\
\leq & C'|t|\sup_{s\in(t,0)}\|{\rm Tr}_s|f|^q\|_{L^{p/q}(G^{-1}(\{s\}),\rho)}^{q/p},
\label{stima_norma_lq}
\end{align}
where
\begin{align*}
C':=\left\|\frac{1}{\|\nabla_HG\|_H}\right\|_{L^\infty(G^{-1}((-\delta,0))}(1\wedge \sup_{s\in(t,0)}\rho(G^{-1}(\{s\}))^{(p-1)/p}),    
\end{align*}
and we have used the fact that $\sup_{s\in(-\delta,0)}\rho(G^{-1}(\{s\}))<+\infty$ (see estimate \eqref{stima_rho_livelli_G}).
By replacing estimate \eqref{stima_traccia_q_int} in \eqref{stima_norma_lq}, it follows that 
\begin{align}
\|f\|^q_{L^q(G^{-1}((t,0),\gamma)}
\leq & C_1|t|\sup_{s\in(t,0)}(\|f\|_{L^{p}(G^{-1}((s,0)),\gamma)}^{p-1}\|\nabla_{H}f\|_{L^{p}(G^{-1}((s,0)),\gamma; H)} +\|f\|_{L^{p}(G^{-1}((s,0)),\gamma)}^{p}) \notag \\
\leq & C_1|t|(\|f\|_{L^{p}(G^{-1}((t,0)),\gamma)}^{p-1}\|\nabla_{H}f\|_{L^{p}(G^{-1}((t,0)),\gamma; H)} +\|f\|_{L^{p}(G^{-1}((t,0)),\gamma)}^{p}),
\label{stima_lq_lp}
\end{align}
with $C_1:=C'\widetilde C$. Letting $q\rightarrow p$ in the left-hand side of \eqref{stima_lq_lp} we infer that
\begin{align}
\|f\|^p_{L^p(G^{-1}((t,0),\gamma)}
\leq C_1|t|(\|f\|_{L^{p}(G^{-1}((t,0)),\gamma)}^{p-1}\|\nabla_{H}f\|_{L^{p}(G^{-1}((t,0)),\gamma; H)} +\|f\|_{L^{p}(G^{-1}((t,0)),\gamma)}^{p}),    
\label{stima_lp_lp}    
\end{align}
We set $\delta_0:=\delta\wedge \frac12 C_1^{-1}$. Then, for every $t\in(-\delta_0,0)$ we get
\begin{align*}
\|f\|^p_{L^p(G^{-1}((t,0),\gamma)}
\leq C_1|t|\|f\|_{L^{p}(G^{-1}((t,0)),\gamma)}^{p-1}\|\nabla_{H}f\|_{L^{p}(G^{-1}((t,0)),\gamma; H)} +\frac12\|f\|_{L^{p}(G^{-1}((t,0)),\gamma)}^{p},        
\end{align*}
which gives
\begin{align}
\|f\|^p_{L^p(G^{-1}((t,0),\gamma)}
\leq 2C_1|t|\|f\|_{L^{p}(G^{-1}((t,0)),\gamma)}^{p-1}\|\nabla_{H}f\|_{L^{p}(G^{-1}((t,0)),\gamma; H)}.
\label{stima_lp_lp_seconda}
\end{align}
By dividing both the sides of \eqref{stima_lp_lp_seconda} by $\|f\|_{L^p(G^{-1}((t,0),\gamma)}^{p-1}$ (which we assumed different from $0$), we infer that
\begin{align*}
\|f\|_{L^p(G^{-1}((t,0),\gamma)}
\leq 2C_1|t|\|\nabla_{H}f\|_{L^{p}(G^{-1}((t,0)),\gamma; H)},   
\end{align*}
and the thesis is proved.

\end{proof}

The next result (based on \cite[Prop. 4.10]{Cel}) is the infinite-dimensional
version of a well-known theorem (see e.g. \cite[Thm. 5.5.2]{Eva}).
We recall that the space $W_{0}^{1,p}(O,\gamma)$ was defined in Definition
\ref{def:Sobolev_Dir}, and characterized as the closure of $\mathcal H^1_0(O)$ in $W^{1,p}(O,\gamma)$ in Lemma \ref{lem:dens_H_funct_sob_dir}.
\begin{thm}
\label{thm:approximation:W_0}Let $O=G^{-1}((-\infty,0))$ with $G$ satisfying Hypotheses \ref{claim:regularity},
and let $f\in W^{1,p}(O,\gamma)$ for some $p\in(1,+\infty)$. then, the following are equivalent:

\begin{enumerate}
\item[i)]$f\in W_{0}^{1,p}(O,\gamma)$;

\item[ii)] $\Tr_0 f\equiv0$;

\item[iii)]the trivial extension $\overline{f}$ of $f$ belongs to $ W^{1,p}(X,\gamma)$. 

\end{enumerate}\end{thm}
\begin{proof}
The points ii) and iii) are equivalent by Lemma \ref{lem:Trace_0-1}.

By definition, if $f\in W_{0}^{1,p}(O,\gamma)$, then it is the limit
of a sequence $(f_{n})\subseteq {\rm Lip}_c(O)$; clearly $f_{n}$ can be extended
as $0$ out of $O$, and the sequence of trivial extensions $(\overline{f}_{n})$ converges
to the trivial extension $\overline{f}\in W^{1,p}(X,\gamma)$ of $f$ as $n\rightarrow +\infty$. This gives $i)\Rightarrow iii)$.

\vspace{2mm}
Now We prove $iii)\Rightarrow i)$. 
Let us set
\[
O_{m}:=G^{-1}\left(\left(-\frac{2}{m},\frac{2}{m}\right)\right), \qquad m\in\N.
\]
$(O_m)$ is a sequence of open decreasing sets such that 
\begin{equation}
\bigcap_{m\in\N}O_{m}=G^{-1}(0).\label{eq:intersections}
\end{equation}

Let $\eta\in C^\infty(\R)$ be such that $0\leq\eta\leq1$,
 $\eta=1$ in $(-\infty,-1]$, $\eta=0$ in $[-1/2,+\infty)$, and
$\eta'\leq0$ everywhere. For every $m\in\N$, we
define $\chi_{m}$ as 
\[
\chi_{m}(x)=\begin{cases}
\eta(mG(x)+1), & \mbox{if }x\in O,\\
0, & \mbox{if }x\notin O.
\end{cases}
\]
It follows that $\chi_m\in{\rm Lip}_{c,H}(O)$ for every $m\in\N$, that $\chi_{m}=1$ on $G^{-1}\left((-\infty,-\frac{2}{m}]\right)$
and $\chi_m=0$ on $G^{-1}\left((-\frac{1}{m},+\infty)\right)$. Further,
\begin{equation}
\nabla_{H}\chi_{m}(x)=(m\eta'(mG(x)+1))\nabla_{H}G(x), \qquad \gamma\textup{-a.e. }\ x\in O,
\label{eq:grad_chi}
\end{equation}
$\nabla_{H}\chi_{m|X\backslash O_{m}}=0$ and
\begin{equation}
\|\nabla_{H}\chi_{m}\|_{L^{\infty}(X;H)}\leq Cm,\label{eq:grad_chi_ineq}
\end{equation}
for some positive constant $C$ independent of $m$. Let $f\in W^{1,p}(O,\gamma)$ be such that ${\rm Tr}_0f=0$. We have
\[
\int_{O}|f|^{p}\left\Vert \nabla_{H}\chi_{m}\right\Vert {}_{H}^{p}\ d\gamma\leq C^{p}m^{p}\int_{O_{m}}|f|^{p}\ d\gamma,
\]
and from Lemma \ref{lem:pseudo_Poinc} it follows that
\begin{equation}
\int_{O}|f|^{p}\left\Vert \nabla_{H}\chi_{m}\right\Vert {}_{H}^{p}\ d\gamma\le C_{0}\|f\|_{W^{1,p}(O_{m},\gamma)}^{p},\label{eq:phi_square_grad_chi}
\end{equation}
for some positive constant $C_0$ independent of $m$ and $\varphi$.

Now we prove that $f\in W_{0}^{1,p}(O,\gamma)$,
by finding a sequence $(f_{m})\subseteq\mbox{Lip}_{c,H}(O)$ which converges to $f$ in $W^{1,p}(O,\gamma)$ as $m\rightarrow+\infty$.
We know that there exists a sequence $(g_{n})\subseteq\mathcal{F}C_{b}^{\infty}(X)$ such that $g_{n}\rightarrow \overline f$ in $W^{1,p}(X,\gamma)$ as $n\rightarrow+\infty$.
We fix $n\in\N$. Then, 
\[
\int_{O}|g_{n}\chi_{m}-f|^{p}\ d\gamma\leq2^{p-1}\left(\int_{O}|g_{n}-f|^{p}\ d\gamma+\int_O|f|^p|\chi_m-1|^pd\gamma \right),
\]
and the right-hand side converges to $0$ as $m\rightarrow\infty$, because $\chi_{m}\rightarrow \u_{O}$ in ${L^{r}(X,\gamma)}$ as $m\rightarrow +\infty$ for every $r\in(1,+\infty)$ by the dominated convergence theorem, and $g_{n}\rightarrow f$ in $L^p(O,\gamma)$ as $n\rightarrow +\infty$. Further, for every $n,m\in\N$ we have
\begin{align*}
& \int_{O}\left\Vert \nabla_{H}(g_{n}\chi_{m})-\nabla_{H}f\right\Vert {}_{H}^{p}\ d\gamma \\
= & \int_{O}\left\Vert g_{n}\nabla_{H}\chi_{m}+\chi_{m}\nabla_{H}g_{n}-\nabla_{H}f-f\nabla_{H}\chi_{m}+f\nabla_{H}\chi_{m}-\chi_{m}\nabla_{H}f+\chi_{m}\nabla_{H}f\right\Vert {}_{H}^{p}\ d\gamma \\
\leq & 4^{p-1}\bigg(\int_{O}(g_{n}-f)^{p}\left\Vert \nabla_{H}\chi_{m}\right\Vert {}_{H}^{p}\ d\gamma+\int_{O}\chi_{m}^{p}\left\Vert \nabla_{H}g_{n}-\nabla_{H}f\right\Vert {}_{H}^{p}\ d\gamma \\
& +\int_{O}(\chi_{m}-1)^{p}\left\Vert \nabla_{H}f\right\Vert {}_{H}^{p}\ d\gamma+\int_{O}f^{p}\left\Vert \nabla_{H}\chi_{m}\right\Vert {}_{H}^{p}\ d\gamma\bigg) \\
\leq & 4^{p-1}\bigg(m^{p}\left\Vert g_{n}-f\right\Vert {}_{L^{p}(O_{m},\gamma)}^{p}+\left\Vert \nabla_{H}g_{n}-\nabla_{H}f\right\Vert {}_{L^{p}(O,\gamma,H)}^{p}
+\int_{O_{m}}\left\Vert \nabla_{H}f\right\Vert {}_{H}^{p}d\gamma \\
& +\int_{O}f^{p}\left\Vert \nabla_{H}\chi_{m}\right\Vert {}_{H}^{p}d\gamma\bigg) \\
\leq & 4^{p-1}(m^{p}\left\Vert g_{n}-f\right\Vert {}_{L^{p}(O_{m},\gamma)}^{p}+\left\Vert \nabla_{H}g_{n}-\nabla_{H}f\right\Vert {}_{L^{p}(O,\gamma,H)}^{p}+(1+C_{0})\left\Vert f\right\Vert {}_{W^{1,p}(O_{m},\gamma)}^{p}),
\end{align*}
where in the last inequality we exploit \eqref{eq:phi_square_grad_chi}. By recalling that $g_{n}\rightarrow f$ in $W^{1,p}(X,\gamma)$ for
$n\rightarrow+\infty$, it follows that, for every fixed $m\in\N$, there
exists $n_{m}\geq m$ such that $\|g_{n_m}-f\|_{W^{1,p}(X,\gamma)}^p\leq m^{-p-1}$, which gives
\begin{align}
\int_{O}\left\Vert \nabla_{H}(g_{n_{m}}\chi_{m})-\nabla_{H}f\right\Vert {}_{H}^{p}\ d\gamma\leq 4^{p-1}(m^{-1}+(1+C_{0})\left\Vert f\right\Vert {}_{W^{1,p}(O_{m},\gamma)}^{p}).
\label{teo_fin_ultima_stima}
\end{align}
We recall that $\gamma(G^{-1}(\{0\}))=0$ by Remark \ref{rem:null-boundary}. Hence, from (\ref{eq:intersections}) the last addend in the right-hand side of \eqref{teo_fin_ultima_stima} converges to $0$ as $m\rightarrow\infty$. 
We now define $f_{m}=g_{n_{m}}\chi_{m}$ for every $m\in\N$. Then, $(f_{m})$ converges to $f$
in $W^{1,p}(O,\gamma)$ as $m\rightarrow\infty$, and it is not hard to show that $f_m\in {\rm Lip}_{c,H}(O)$ for every $m\in\N$.
\end{proof}

\section{Examples \label{sec:Examples}}

\subsection{Region below graphics}

As above, we consider
a basis $\{h_{i}\}_{i\in\N}$ in $R_{\gamma}(X^{*})$. In the following, we will denote $\pi_{\hat{h}_1}$  with $\pi_1$. We will define
a function $G$ such that $O=G^{-1}((-\infty,0))$ is the region below the graph of a smooth function.

Let $\Phi$ be a real-valued function on $X$ such that $\partial_{h_{1}}(\Phi)\equiv0$. Hence, for every $x$ we have $\Phi(x)=\Phi(x-\pi_{1}(x))$.
We set
\[
G(x)=\hat{h}_{1}(x)-\Phi(x), \qquad x\in X.
\]
In this case, $O$ is just the region below the graph of $\Phi$. 
We assume that $\Phi$ is continuous and also satisfies the following conditions:
\begin{enumerate}
    \item $\Phi\in {\rm Lip}_{H}(X)$;
    \item $\Phi\in W^{2,p}(X,\gamma)$ for some $p>1$ and $\|D^2_H\Phi\|_{\mathcal L_2(H)}$ is essentially bounded;
    \item $\Phi+L\Phi\in L^\infty(X,\gamma)$.
\end{enumerate}

Under these assumptions, it follows that
\begin{align*}
\nabla_{H}G(x)& =h_{1}-\nabla_{H}\Phi(x-\pi_{1}(x)), \\
D_{H}^{2}G(x) & =-D_{H}^{2}\Phi(x-\pi_{1}(x)), \\
LG(x) & =-\hat{h}_{1}(x)-L\Phi(x-\pi_{1}(x)),
\end{align*}
$\gamma$-a.e. $x\in X$, and $G$ satisfies Hypotheses \ref{claim:regularity}. Indeed, we have
\begin{align*}
LG(x)
= & -\hat h_1(x)- L\Phi(x-\pi_{1}(x))+\Phi(x-\pi_1(x))-\Phi(x-\pi_1(x)) \\
= & (-\hat h(x)+\Phi(x-\pi_1(x)))
-(\Phi(x-\pi_1(x))+L\Phi(x-\pi_1(x))) \\
= & G(x)-(\Phi(x-\pi_1(x))+L\Phi(x-\pi_1(x))),
\end{align*}
$\gamma$-a.e. $x\in X$, and both the addends are bounded on $G^{-1}(I_\delta)$.
If we take for instance $\Phi=c\in\R$, we get that open half-planes satisfy our assumptions.

\subsection{Brownian motion and pinned Brownian motion}

For the following examples we refer to \cite[Section 5]{Dap}. 

We recall (see \cite[Section 2.3]{Bog}) that a Brownian motion starting
from 0 can be modelled by a Wiener space $(X,\gamma_{W})$ where
$X=L^{2}(0,1)$ (with Lebesgue measure), and $\gamma_{W}$ concentrates
on the set of the elements of $L^{2}(0,1)$ which have a continuous
representative $f$ such that $f(0)=0$. The Cameron-Martin space
$H$ is the set of the elements of $L^{2}(0,1)$ which have an absolutely continuous representative $f$ such that $f'\in L^{2}(0,1)$ and $f(0)=0$. Finally, for every $f_{1},f_{2}\in H$ the inner product in $H$ is defined as
$\left\langle f_{1},f_{2}\right\rangle _{H}=\int_{0}^{1}f_{1}'(s)f_{2}'(s)\ ds$.
In the following, for every $h\in H$, we will identify $h$ with
its absolutely continuous representative.

We define an orthonormal basis $\{e_{n}\}_{n\in\N}$ of $L^{2}(0,1)$ as 
\[
e_{n}(s)=\sqrt{2}\sin\Bigl(\frac{s}{\sqrt{\lambda_{n}}}\Bigr)=\sqrt{2}\sin\Bigl(\frac{2n+1}{2}\pi s\Bigr)
\]
where 
\[
\lambda_{n}=\frac{1}{\pi^{2}\left(n+\frac{1}{2}\right)^{2}}.
\]
For every $n\in\N$ we set $h_{n}=\sqrt{\lambda_{n}}e_{n}$. It follows that $\{h_{n}\}_{n\in\N}$ is an orthonormal basis of $H$. 

We consider a function $g\in C^{2}(\R)$, with bounded first and second order derivative, such that there exists $C>0$ such that
\begin{equation}
|g''(\xi)-g''(\eta)|\leq C|\xi-\eta|(|\xi|+|\eta|),\label{eq:example_g_curvature-1}
\end{equation}
for every $\xi,\eta\in\R$. Further, we assume that there exist $\alpha_{1},\alpha_{2},\beta_{1},\beta_{2}\in\R$ such that $|g'(\xi)|\geq a$ (hence $g'(\xi)\neq0$) for every
$\xi\in\R$ and
\begin{equation}
\alpha_{1}g(\xi)+\beta_{1}\leq\xi g'(\xi)\leq\alpha_{2}g(\xi)+\beta_{2}\label{eq:asymptote-1}
\end{equation}
 for every $\xi\in\R$. 

The above assumptions are satisfied, for instance, by the function $g=p/q$, where
$q$ is a positive polynomial of degree $m\in\N$ and $p$ polynomial of
degree $m+1$ such that $p'(\xi)\neq0$ for every $\xi\in\R$.
\begin{prop}
\label{prop:brownian}
Let us assume that $g\in C^2(\R)$ satisfies \eqref{eq:example_g_curvature-1} and \eqref{eq:asymptote-1}, and let $r$ belong to the range
of $g$. We define the function
\[
G(x):=\int_{0}^{1}g(x(s))\ ds-r
\]
for every $x\in X=L^{2}(0,1)$. Then, $G$ satisfies
Hypotheses \ref{claim:regularity}.
\end{prop}
\begin{proof}

It is not hard to show that $G$ is $H$-differentiable. For every $h\in H$ and every $x\in X$ we have
\[
\left\langle \nabla_{H}G(x),h\right\rangle _{H}=\int_{0}^{1}g'(x(s))h(s) ds
\]
and 
\[
\|\nabla_{H}G(x)\|_{H}\leq\sqrt{\int_{0}^{1}|g'(x(s))|^{2} ds}\leq\|g'\|_{\infty}.
\]
Moreover, for every $x,y\in X$, 
\begin{align*}
\|\nabla_{H}G(x)-\nabla_{H}G(y)\|_{H}^{2}
\leq & \int_{0}^{1}|g'(x(s))-g'(y(s))|^{2}ds 
\leq  \int_{0}^{1}\| g''\|_{\infty}^{2}|x(s)-y(s)|^{2} ds \\
\leq & \|g\| _{\infty}^{2}\|x-y\|_{X}^{2},
\end{align*}
from which it follows that $G\in\mathcal H^1(X)$. Further, $D_{H}^{2}G$ is everywhere defined and, for every $h,k\in H$ and every $x\in X$ we get
\[
\left\langle \left(D_{H}^{2}G(x)\right)(h),k\right\rangle _{H}=\int_{0}^{1}g''(x(s))h(s)k(s) ds.
\]
Hence,
\begin{align*}
\|D^2_HG(x)\|_{\mathcal L_2(H)}
= & \sum_{n=1}^\infty\|D^2_HG(x)h_n\|_H^2
\leq\sum_{n=1}^\infty\lambda_n\|g''\|_{\infty}^2<+\infty,
\end{align*}
for every $x\in X$.
Let us consider the function $\bar{h}\in H$ defined by $\bar h(s)=s$ for every $s\in[0,1]$. It follows that $\bh>0$ and $\|\bh\|_{H}=1$. Further, since $g'$ has constant sign (from $|g'|\geq a$), we have 
\begin{align*}
|\left\langle \nabla_{H}G(x),\bh\right\rangle _{H}|
= &\int_{0}^{1}|g'(x(s))|\bh(s)\ ds
\geq a\int_{0}^{1}\bh(s)\ ds
= \frac{a}{2},
\end{align*}
for every $x\in X$, which implies that
\begin{equation}
\|\nabla_{H}G\|_{H}^{-1}\leq \frac{2}{a}.
\label{eq:grad_inverse-1}
\end{equation}

If we consider the sequence $\{h_{k}=\sqrt {\lambda_k}e_k\}_{k\in\N}$, then the
series $\sum_{k=1}^{\infty}h_{k}^{2}$ uniformly converges to a function $f\in C([0,1])$. Moreover, 
\begin{align*}
L G(x)
= & \sum_{i=1}^{\infty}\langle D_{H}^{2}G(x)(h_{i}),h_{i}\rangle_H-\sum_{i=1}^{\infty}\left\langle \nabla_{H}G(x),h_{i}\right\rangle _{H}\widehat{h_i}(x) \\
=& \sum_{i=1}^{\infty}\int_{0}^{1}g''(x(s))h_{i}^{2}(s)ds-\int_{0}^{1}g'(x(s))x(s)ds \\
= & \int_{0}^{1}g''(x(s))f(s)ds-\int_{0}^{1}g'(x(s))x(s)ds.
\end{align*}
The first addend in the last right-hand side of the above chain of equality is bounded because $g''\in C_b(\R)$ and $f\in C([0,1])$. Further, from (\ref{eq:asymptote-1}) we infer that
\[
\int_{0}^{1}g'(x(s))x(s) ds\geq\int_{0}^{1}\left(\alpha_{1}g(x(s))+\beta_{1}\right) ds=\alpha_{1}G(x)-\alpha_{1}r+\beta_{1}
\]
and 
\[
\int_{0}^{1}g'(x(s))x(s)\leq\int_{0}^{1}\left(\alpha_{2}g(x(s))+\beta_{2}\right) ds=\alpha_{2}G(x)-\alpha_{2}r+\beta_{2}.
\]
Therefore, $LG$ is bounded in $G^{-1}((-\delta,\delta))$ for every
$\delta>0$.

Finally, since $r$ belongs to range of $g$ it follows that that $G^{-1}(\{0\})\neq\varnothing$, we conclude that $G$ fulfills Hypotheses \ref{claim:regularity}.
\end{proof}
An analogous example can be provided for pinned Wiener space, which models
Brownian bridge with starting point at $0$ and subject to the condition
that in $1$ the arriving point is $0$. $(X,\tilde{\gamma}_{W})$ where
$X=L^{2}(0,1)$, the Cameron-Martin space is $H=W_{0}^{1,2}(0,1)$.
We recall that $\{e_{n}\}_{n\in\N}$ with $e_n=\sqrt 2\sin(n\pi\cdot )$ 
for every $n\in\N$ is an orthonormal basis of $X$,
and $\{h_n\}_{n\in\N}$, where $h_{n}=\sqrt{2}\pi^{-1}n^{-1}\sin(n\pi\cdot)$ for every $n\in\N$ is an orthonormal basis of $H$.

If we consider a function $g\in C^2(\R)$ which enjoys \eqref{eq:example_g_curvature-1} and \eqref{eq:asymptote-1}, arguing as in the proof of Proposition \ref{prop:brownian} we gain the following result.
\begin{prop}
Given $r$ in the range of
$g$, we define
\[
G(x)=\int_{0}^{1}g(x(s))\ ds-r
\]
for every $x\in X=L^{2}(0,1)$.
Then, the function $G$  satisfies Hypotheses
\ref{claim:regularity}.
\end{prop}


\begin{thebibliography}{10}
\bibitem{Add0}
D. Addona, G. Cappa, S. Ferrari, 
\emph {Domains of elliptic operators on sets in Wiener space}, Infin. Dimens. Anal. Quantum Probab. Relat. Top. {\bf23} (2020), 2050004, 42 pp.

\bibitem{Add2}
D. Addona, G. Menegatti, M. Miranda Jr,
\emph{BV functions on open domains: the Wiener case and a Fomin differentiable case},
Commun. Pure Appl. Anal. {\bf 19} (2020), 2679-2711.

\bibitem{Add}
D. Addona, G. Menegatti, M. Miranda Jr, 
\emph{On integration
by parts formula on open convex sets in Wiener spaces}, 
J. Evol. Equ., {\bf21} (2021), 1917-1944.

\bibitem{Bon}
S. Bonaccorsi, L. Tubaro, M. Zanella, 
\emph{Surface measures and integration by parts formula on levels sets induced by functionals of the Brownian motion in $\R^n$}, 
NoDEA Nonlinear Differential Equations Appl., \textbf{27} (2020), 22 pp.

\bibitem{Bog}
V. I. Bogachev,
{ Gaussian Measures,} Mathematical
Surveys and Monographs, American Mathematical Society, 1998.

\bibitem{Bog1}
V. I. Bogachev, A. Y. Pilipenko, A. V. Shaposhnikov, 
\emph{Sobolev functions on infinite-dimensional domains}, J. Math. Anal. Appl. {\bf419} (2014), 1023-1044.

\bibitem{Cel}
P. Celada, A. Lunardi, 
\emph{Traces of Sobolev functions on
regular surfaces in infinite dimensions}, 
{J. Funct. Anal.,}
{\bf266} (2014), 1948-1987.

\bibitem{Dap}
G. Da Prato, A. Lunardi,
\emph{Maximal $L^2$ regularity regularity for Dirichlet problems in Hilbert spaces}, 
{J. Math. Pures Appl., {\bf99}} (2013), 741-765.

\bibitem{Die}
J. Diestel, J. J. Uhl, 
{Vector measures}, Mathematical
Surveys and Monographs, 15, American Mathematical Society, 1977.

\bibitem{Eva}
L. Evans, 
{Partial Differential Equations}. American
Mathematical Society, 1998.

\bibitem{Fey}
D. Feyel,
\emph{Hausdorff-Gauss Measures}, 
in: {Stochastic
Analysis and Related Topics, VII.}, Progr. in Probab. 98, Birkh\"auser, 2001, 59-76.

\bibitem{Hin}
M. Hino, 
\emph{Dirichlet spaces on H-convex sets in Wiener
space}, 
{Bull. Sci. Math.}, {\bf135} (2011) 667-683; Erratum:
Bull. Sci. Math., {\bf137} (2013) 688-689.

\bibitem{Hin4}
M. Hino. 
\emph{On Dirichlet spaces over convex sets in infinite dimensions},
Finite and Infinite Dimensional Analysis in Honor
of Leonard Gross, Contemp. Math., {\bf317} (2003), 143-156.

\bibitem{Su98} 
H. Sugita,
\emph{Positive generalized Wiener functions and potential theory over abstract Wiener spaces},
 Osaka J. Math., {\bf 25} (1988), 665-696.
\end{thebibliography}
\end{document}